\documentclass{amsart}

\usepackage{amssymb, amsmath}
\usepackage[a4paper,top=3cm,bottom=3cm,left=1.5 cm,right=1.5cm]{geometry}

\usepackage[T1]{fontenc}
\usepackage[utf8]{inputenc}
\usepackage[english]{babel}
\usepackage{lipsum}
\usepackage[dvips]{graphicx}
\usepackage{alltt,graphicx,graphpap,color}
\usepackage{latexsym}
\usepackage{amsxtra}
\usepackage{amssymb}
\usepackage{amsmath}
\usepackage{amscd}
\usepackage{amsthm}
\usepackage{mathrsfs}
\usepackage{mathtools}
\usepackage{quoting}
\usepackage{indentfirst}
\usepackage{verbatim}
\quotingsetup{font=small}
\usepackage{ytableau}
\usepackage{xcolor}

\DeclareMathOperator{\I}{Id}

\newtheorem{Theorem}{Theorem}[section]
\newtheorem{Definition}{Definition}[section]

\newtheorem{Remark}{Remark}[section]
\newtheorem{Proposition}{Proposition}[section]
\newtheorem{Lemma}{Lemma}[section]
\newtheorem{Corollary}{Corollary}[section]
\usepackage{amssymb, amsmath}

\title{A characterization of varieties of algebras of proper central exponent equal to two}

\author{Francesca S. Benanti}
\address{Dipartimento di Matematica e Informatica, Universit\`a di Palermo, Via Archirafi 34, 90123, Palermo, Italy}
\email{francescasaviella.benanti@unipa.it}

\author{Angela Valenti}
\address{Dipartimento di Ingegneria, Universit\`a di Palermo, Viale delle Scienze, 90128, Palermo, Italy}
\email{angela.valenti@unipa.it}

\thanks{Both authors are partially supported by INDAM-GNSAGA of Italy
	and  by  Fondo Finalizzato alla Ricerca dell'Università degli Studi di Palermo  (FFR) – Anno 2024}



\begin{document}
	
	\begin{abstract}
		Let $F$ be a field of characteristic zero and let $ \mathcal V$ be a variety of associative
		$F$-algebras.  In \cite{regev2016} Regev introduced a numerical sequence measuring the growth of the proper central polynomials of a generating algebra of $ \mathcal V$. Such sequence $c_n^\delta(\mathcal V), \, n \ge 1,$ is called the sequence of proper central polynomials of $ \mathcal V$ and in \cite{GZ2018}, \cite{GZ2019} the authors computed its exponential growth. This is an invariant of the variety. They also showed that  $c_n^\delta(\mathcal V)$ either grows exponentially or is polynomially bounded.
		
		The purpose of this paper is to characterize the varieties of associative
		algebras whose
		exponential growth of $c_n^\delta(\mathcal V)$ is greater than two. As a consequence, we find a characterization
		of the varieties whose corresponding exponential growth is equal to two.
	\end{abstract}
	
	\maketitle

	\section{Introduction}
	Let $A$ be an algebra over a field  $F$ of characteristic zero and  $F\langle X\rangle$ the free associative  algebra freely generated over $F$ by  the set
	$X=\{x_1,x_2,\ldots \}$.
	A  polynomial $f=f(x_1, \ldots , x_n)\in F\langle X\rangle$ is a central polynomial of $A$  if, for any $a_1, \ldots , a_n \in A$,  $f(a_1, \ldots , a_n)\in
	Z(A),$
	the center of $A$. In case
	$f$ takes a non-zero value in $A$, we say that $f$ is a proper central polynomial,
	otherwise we say that $f$ is a polynomial identity of $A$.
	We denote by Id$(A)$ the $T$-ideal of $F\langle X\rangle$ consisting of the polynomial identities of $A$
	and by Id$^z(A)$ the $T$-space of central polynomials of $A$.
	
	Central polynomials were first studied after a famous conjecture of Kaplansky asserting that  $M_k(F)$, the algebra of $k\times k$ matrices over $F$,  has
	proper central polynomials (see \cite{Kaplansky}). Later on, such a conjecture was proved independently by Formanek in \cite{Formanek}  and Razmyslov in
	\cite{Razmyslov}.
	This result is highly non-trivial since even if an algebra has a non-zero center, the existence of proper central
	polynomials is not granted. For instance, it is well-known (see \cite{GZ2018}) that
	the algebra of upper block triangular matrices has no proper central polynomials.
	
	Here we are interested in a quantitative approach in order to get information about the growth of the proper central polynomials of a given  algebra. To this
	end, for any
	$n \ge 1$  let $P_n$ be the space of multilinear polynomials in $x_1, \ldots, x_n$.
	Then define $P_n(A)$ as $P_n$ modulo  the identities of $A,$ and  $P_n(A)^z$ as $P_n$ modulo the central polynomials of $A$.
	The corresponding dimensions $c_n(A)=\dim P_n(A)$ and $c_n^z(A)=\dim P_n(A)^z$ are called the $n$-th codimension
	and the $n$-th central codimension of $A$, respectively.
	Clearly $c_n(A)- c_n^z(A)= c_n^\delta(A)$ is the dimension of the corresponding space of proper central polynomials.
	
	The asymptotics of the codimensions of an algebra $A$ have been extensively studied in the past, see for instance \cite{reg1}, \cite{AJK}, \cite{B}, \cite{BR}
	or
	\cite{GZbook} and \cite{AGPR} for a comprehensive overview of the main results.
	It turns out that if $A$ satisfies a non-trivial identity, then $\lim_{n\to \infty} \sqrt[n]{c_n(A)}=exp(A)$ exists, and is an integer called the PI-exponent
	of $A$ (see \cite{GZ2}). Recently in \cite{GZ2018} and \cite{GZ2019} it was proved that also $\lim_{n\to \infty}  \sqrt[n]{c_n^z(A)} =exp^z(A)$
	and $\lim_{n\to \infty}  \sqrt[n]{c_n^\delta(A)} =exp^\delta(A)$ exist,
	are integers  called the central exponent and the proper central exponent of $A$, respectively. Moreover all three codimensions either grow exponentially or
	are polynomially bounded.
	When $exp(A)\ge 2$,   $exp(A)$ and  $exp^z(A)$ coincide, whereas
	it can be shown that the difference  $exp(A) -exp^\delta(A)$ can be any positive integer (see \cite{GZ2018}).
	Hence $exp^\delta(A)$ is a second invariant of the $T$-ideal $\I(A)$ that is worth investigating.
	
	Let $G$ denote the infinite dimensional Grassmann algebra and $UT_2(F)$  the algebra of $2\times 2$ upper triangular matrices. Kemer in \cite{kemer0} proved
	that if an algebra $A$ is such that $exp(A)\ge 2$, then either  $\I(G)\supseteq \I(A)$ or $\I(UT_2(F)) \supseteq \I(A)$. Since
	it is well-known that $exp(G)=exp(UT_2)=2$ (see \cite{GZbook}) and no intermediate growth is allowed,
	it follows that if $B$ is an algebra such that $\I(B)\supsetneq \I(G)$ or $\I(B) \supsetneq \I(UT_2(F))$ then the codimensions of $B$ grow polynomially.
	In the language of varieties of algebras, this says that $G$ and $UT_2(F)$
	generate the only two varieties of almost polynomial growth.
	
	Recently in  \cite{GLP} an analogous result was  proved  in the setting of proper central polynomials. The authors
	constructed two finite dimensional algebras named $D$ and $D_0$, having exponential growth of the proper central codimensions and
	classified up to PI-equivalence, the algebras $A$ for which the sequence
	$c_n^{\delta}(A)$, $n\ge 1$, has almost polynomial growth.
	It  follows that $G,$  $D$ and $D_0$ generate the only varieties of almost polynomial growth of the proper central polynomials.
	
	In \cite{GZ2000},  Giambruno and Zaicev characterized the varieties of PI-exponent equal to two by exhibiting a
	list of five suitable algebras of exponent 3 or 4.
	
	The purpose of this paper is to characterize the varieties having proper central exponent
	greater than two. To this end, we shall explicitly exhibit a  list of nine
	algebras $\mathcal{A}_i$ in order to prove the following result: a variety $\mathcal{V}$ has proper central exponent
	greater than two if and only if $\mathcal{A}_i \in \mathcal{V},$ for some $i$. By putting together this theorem
	with the results in \cite{GLP},
	we shall obtain a characterization of the varieties of proper central exponent equal to two.
	
	\bigskip

	\section{Preliminaries} \label{section2}

	Throughout this paper all algebras will be associative and over a field $F$ of characteristic zero.
	Unless otherwise stated, we shall also assume that $F$ is algebraically closed.
	Let
	$F\langle X\rangle$ be the free algebra over $F$ on a countable set $X=\{x_1, x_2, \ldots\}$.
	Let Id$(A)$ be the $T$-ideal of polynomial identities of the $F$-algebra $A$ and
	let Id$^z(A)$ be the space of central polynomials of $A.$ Notice that Id$^z(A)$
	is a $T$-space, i.e., invariant under all endomorphisms of $F\langle X\rangle.$
	Clearly Id$(A) \subseteq$ Id$^z(A)$ and the proper central polynomials correspond to the
	quotient space Id$^z(A)$/Id$(A)$.

	Regev in \cite{regev2016} introduced the notion of central codimensions as follows.
	Let
	$P_n$ be the space of  multilinear polynomials in $x_1, \ldots, x_n$. Then,
	for $n=1, 2, \ldots$,
	$$
	c_n(A)= \dim \frac{P_n}{P_n\cap \I(A)}, \, \,
	c_n^z(A)= \dim \frac{P_n}{P_n\cap \I^z(A)},  \, \,
	c_n^\delta(A)= \dim \frac{P_n\cap \I^z(A)}{P_n\cap \I(A)}
	$$
	are
	the sequences of codimensions, central codimensions and proper central codimensions of $A$, respectively.

	It was proved in \cite{GZ1}, \cite{GZ2}, \cite{GZ2018} and \cite{GZ2019} that for any PI-algebra $A$ the following three limits
	$$
	exp(A)=\lim_{n\to\infty}\sqrt[n]{c_n(A)},\, \, exp^z(A)=\lim_{n\to\infty}\sqrt[n]{c_n^z(A)}, \, \, exp^\delta(A)=\lim_{n\to\infty}\sqrt[n]{c_n^\delta(A)}
	$$
	exist and are non negative integers called the PI-exponent, the central exponent and the proper central exponent (called also $\delta$-exponent) of the algebra
	$A$, respectively.
	Clearly $exp^z(A), exp^\delta(A)\le exp(A)$ and in \cite[Theorem 3]{GZ2019} it was shown that  if $\exp(A)\ge 2$, then $exp^z(A)= exp(A)$.
	A basic property of the above three numerical sequences is that they either grow exponentially or are polynomially bounded.
	Hence no intermediate growth is allowed.

	Let us introduce a useful tool in the theory of polynomial identities, the
	Grassmann envelope of an algebra.
	Let $G$ be the infinite dimensional Grassmann  algebra, that is the algebra with $1$ generated by countably many elements $e_1, e_2, \ldots$,
	subject to the condition  $e_ie_j=-e_je_i$. $G$ has a natural structure of superalgebra $G=G^{(0)}\oplus G^{(1)}$ where
	every monomial in the $e_i$'s of even (odd resp.) length has homogeneous degree zero (one resp.).
	
	If $B= B^{(0)}\oplus B^{(1)}$ is a superalgebra, one can construct the algebra   $G(B)=( G^{(0)} \otimes B^{(0)}) \oplus ( G^{(1)} \otimes B^{(1)})$
	which is called the Grassmann envelope of $B$.
	The importance of the Grassmann envelope is highlighted in  a well-known result of Kemer (see \cite{kemer}) stating that if $A$ is any PI-algebra, there exists
	a finite dimensional superalgebra $B$ such that $\I(A)=\I(G(B))$.

	It is worth noticing that if $A$ is a finite dimensional algebra, then regarding $A$ as an algebra with trivial grading we get
	$\I(G(A))= \I( G^{(0)} \otimes A)=\I(A)$,  and we may work with $A$ instead of $G(A)$.

	Now we recall how to compute the exponent and the proper central exponent
	of the algebra $G(B)$.
	By the Wedderburn-Malcev decomposition we write $B=\bar B +J$  where $\bar B$ is a maximal semisimple superalgebra inside $B$
	and $J$ is the Jacobson radical of $B$ (see \cite[Theorem 3.4.3]{GZbook}).  Also we can write $\bar B= B_1\oplus \cdots \oplus B_q$, a decomposition into
	simple superalgebras.
	
	Let $B_{i_1}, \ldots, B_{i_k}$ be distinct simple subalgebras of $\bar B$.
	Then we say that $C=B_{i_1}\oplus \cdots \oplus B_{i_k}$ is an admissible subalgebra of $B$
	if $B_{i_1}J \cdots J B_{i_k}\ne 0$. Moreover we say that $C$ is a centrally admissible subalgebra of $B$,
	if there exists a proper central polynomial $f(x_1, \ldots, x_n)$ of $G(B)$ having a non-zero evaluation involving  at least one element of each
	$G(B_{i_h})$, $1\le h \le k$.
	
	Then $exp(G(B))$ is the maximal dimension of an admissible subalgebra of $B$ and $exp^\delta(G(B))$ is the maximal dimension of a
	centrally admissible subalgebra of $B$.

	Recall that if $\mathcal V=var(A)$ is the variety of algebras generated by an algebra $A$, then we write $exp(\mathcal V)=exp(A)$
	and $exp^\delta(\mathcal V)=exp^\delta(A)$. Hence the growth of $\mathcal V$ is the growth of the
	codimensions of the algebra $A$.
	Moreover
	a variety $\mathcal V$ has almost polynomial growth if $\mathcal V$ grows exponentially but any
	proper subvariety grows polynomially.

	By \cite[Remark 1]{GZ2019} if two algebras have the same identities, they also have the same proper central polynomials.
	Hence the proper central exponent of an algebra $A$ is an invariant of its $T$-ideal of identities and also of the corresponding variety of algebras. So we say
	that
	a variety of algebras ${\mathcal V}$  has almost polynomial $\delta$-growth if $exp^\delta(\mathcal V)\ge 2$
	and for any proper subvariety ${\mathcal W}\subsetneq {\mathcal V}$ we have that $exp^\delta(\mathcal W)\le 1$.
	By abuse of notation we also say that an algebra $A$ has almost polynomial $\delta$-growth if $exp^\delta(A)\ge 2$
	and for any algebra $B$ such that $\I(B) \supsetneq \I(A)$, we have that $exp^\delta(B)\le 1$.

	Recently in \cite{GLP} the authors classified the $T$-ideals whose corresponding varieties have almost polynomial growth of the proper central codimensions.
	More precisely,  it was proved that if $A$ is a PI-algebra with $exp^\delta(A)\geq 2$, then Id$(A)\subseteq$ Id$(D)$ or Id$(A)\subseteq$ Id$(D_0)$ or
	Id$(A)\subseteq$ Id$(G)$, where $D=\{(a_{ij})\in UT_3(F) \mid a_{11}=a_{33}\}$ and
	$D_0=\{(a_{ij})\in UT_4(F) \mid a_{11}=a_{44}=0\}$.
	Here $UT_n(F)$ denotes the algebra of $n\times n$ upper triangular matrices over $F$.
	
	\bigskip

	\section{Some algebras of small $\delta$-exponent greater than two}

	The purpose of this section is to introduce some suitable algebras that
	will allow us to prove the main result of this paper.
	We start by recalling some notation and definitions.
	
	For an ordered set of positive integers $(d_1,\ldots,d_m)$ let
	$UT(d_1,\ldots ,d_m)$
	denote the algebra of  upper block-triangular
	matrices over $F$ of size $d_1,\ldots,d_m$, i.e.,
	
	$$
	UT(d_1,\ldots ,d_m)=\begin{pmatrix}M_{d_1}(F) & \cdots  & *  & * \cr
		0 & \ddots & &  \cr
		\vdots &   &  & * \cr
		0 & \cdots & 0 & M_{d_m}(F)
	\end{pmatrix},
	$$
	
	\noindent
	where $M_{d_i}(F)$ is the algebra of $d_i \times d_i$ matrices over
	$F$.
	
	Recall that if  $H$ is an arbitrary abelian group an $H$-grading on $M_n(F)$ is elementary if
	there is an $n$-tuple
	$\bar h =(h_1,\ldots,h_n)$ such that  the matrix units
	$e_{ij}$ are of homegeneous degree $ h_i^{-1}h_j.$
	Clearly the algebra of upper block triangular matrices also admits elementary gradings. In fact, the embedding of such an algebra into a full matrix algebra
	with an elementary grading makes it a homogeneous subalgebra.
	We shall denote by $UT(d_1,\ldots,d_m)_{(h_1,\ldots,h_n)}$ the algebra of upper block triangular matrices with elementary grading induced by the $n$-tuple
	$\bar h=(h_1,\ldots,h_n)$ with $n=d_1+\cdots +d_m.$
	In particular, for $H=\mathbb Z_2$, $UT(d_1,\ldots,d_m)_{(h_1,\ldots,h_n)}$ is a superalgebra and $G(UT(d_1,\ldots,d_m)_{(h_1,\ldots,h_n)})$ denotes its
	Grassmann envelope.

	We also need to recall the classification of the finite dimensional simple superalgebras over $F$ (see \cite{kemer}, pag.21).
	Any such superalgebra is isomorphic to one of the following:
	\begin{enumerate}
		\item[i)]  $M_{k,l}(F)$ is the algebra $M_{k+l}(F)$,  $k\geq l\geq 0$, $k\neq 0$,
		endowed with the grading induced by the $(k+l)$-tuple $(\underbrace{0, \ldots ,0}_{\text{$k$ times}},\underbrace{1, \ldots ,1}_{\text{$l$ times}})$;
		\item[ii)] $M_k(F) + cM_k(F)$, where $c^2=1$, with grading $(M_k(F), cM_k(F))$.
	\end{enumerate}
	We shall denote by $M_{k,l}(G)$ and $M_k(G)$, respectively,  the corresponding Grassmann envelopes.
	In what follows we shall identify $F + cF$ with the graded subalgebra of $M_2(F)_{(0,1)}$ generated by the elements
	$c_0=e_{11}+e_{22},$ and $c_1=e_{12}+e_{21}.$

	\medskip
	
	Next we define the following nine algebras over $F$ that we shall use in the sequel.
	
	\begin{enumerate}
		\item[ ]  $\mathcal{A}_1=M_2(F)$,
		\smallskip
		\item[ ] $\mathcal{A}_2=M_{1,1}(G)$,
		\smallskip
		\item[ ]
		$ \mathcal{A}_3=\{ (a_{ij}) \in UT_4(F) \mid a_{11}=a_{44}\}$,
		
		\medskip
		\item[ ]
		$\mathcal{A}_4= \{ (a_{ij}) \in UT_5(F) \mid a_{11}=a_{55}=0\}$,
		
		\medskip
		
		\item[ ]
		$\mathcal{A}_5= \{(a_{ij}) \in UT_3(G) \mid a_{11}=a_{33},\  a_{22} \in G^{(0)}\}$,

		\medskip
		
		\item[ ]
		$\mathcal{A}_6= \{ (a_{ij}) \in G(UT(2,3)_{(0,0,0,1,0)})\mid a_{11}=a_{21}=a_{53}=a_{54}=a_{55}=0,\ a_{33}=a_{44}, \  a_{34}=a_{43} \}$,

		\medskip
		
		\item[ ]
		$\mathcal{A}_7= \{ (a_{ij}) \in G(UT(3,2)_{(0,0,1,0,0)}) \mid a_{11}=a_{21}=a_{31}=a_{54}=a_{55}=0,\ a_{22}=a_{33}, \ a_{23}=a_{32}\}$,
		
		\medskip
		\item[ ]
		$\mathcal{A}_8=\{(a_{ij})\in G(UT(1,2,1)_{(0,0,1,0)})  \mid a_{11}=a_{44},\ a_{22}=a_{33}, \, a_{23}=a_{32}\}$,

		\medskip
		\item[ ]
		
		\noindent
		$
		\mathcal{A}_9=   \{ (a_{ij}) \in G(UT(1,2,1)_{(0,0,1,1)}) \mid a_{11}=a_{44},\ a_{22}=a_{33},\  a_{23}=a_{32}\}.$

	\end{enumerate}
	
	\bigskip
	It is well known that $exp(M_2(F))=exp^\delta(M_2(F))=4$
	and
	$exp(M_{1,1}(G))=exp^\delta(M_{1,1}(G))
	=4.$ For the remaining algebras we have the following
	
	\begin{Lemma} \label{exponent Ai}
		For $i=3,\ldots , 9$, $exp(\mathcal A_i)=exp^\delta(\mathcal A_i)=3$.
	\end{Lemma}
	
	\begin{proof}
		For the algebra $\mathcal{A}_3$, since $F(e_{11}+e_{44})Fe_{12}Fe_{22}Fe_{23}Fe_{33}\neq 0$, then $F(e_{11}+e_{44})\oplus Fe_{22}\oplus Fe_{33}$ is a maximal
		admissible subalgebra. Hence  $exp(\mathcal{A}_3)=3$.
		The center of $\mathcal{A}_3$ is $Z(\mathcal{A}_3)=F(e_{11}+e_{22}+e_{33}+e_{44})+Fe_{14}$.
		Since $[e_{11}+e_{44},e_{12}][e_{22},e_{23}][e_{33},e_{34}]=e_{14}$ we have that $[x_1,x_2][x_3,x_4][x_5,x_6]$ is a proper central polynomial of
		$\mathcal{A}_3$ and, so, $F(e_{11}+e_{44})\oplus Fe_{22}\oplus Fe_{33}$
		is a maximal centrally admissible subalgebra. It follows that
		$exp^\delta(\mathcal{A}_3)=3$.

		Regarding the algebra $\mathcal{A}_4$,
		it is easy to see that $Fe_{22}\oplus Fe_{33}\oplus Fe_{44}$ is a maximal admissible subalgebra and, so, $exp(\mathcal{A}_4)=3.$
		The center of $\mathcal{A}_4$ is $Z(\mathcal{A}_4)=Fe_{15}$.  Moreover $[x_1,x_2][x_3,x_4][x_5,x_6] [x_7,x_8]$ is a proper central polynomial of
		$\mathcal{A}_4$ since $[e_{12},e_{22}][e_{23},e_{33}][e_{33},e_{34}][e_{44}, e_{45}]=e_{15}$.
		It follows that $Fe_{22}\oplus Fe_{33}\oplus Fe_{44}$ is a maximal centrally admissible subalgebra and so $exp^\delta(\mathcal{A}_4)=3$.
		
		A maximal admissible subalgebra of $\mathcal A_5$ is
		$(F+cF)(e_{11}+e_{33})\oplus Fe_{22}$, hence $exp(\mathcal A_5)=3.$
		The center of $\mathcal A_5$ is $Z(\mathcal A_5)=G^{(0)}(e_{11}+e_{22}+e_{33})+G^{(0)}e_{13}.$
		Moreover $[[x_1,x_2,x_3][x_4,x_5,x_6],x_7]$ is a proper central polynomial  of $\mathcal A_5$ and, so, $exp^\delta(\mathcal A_5)=3$.
		
		It is easy to check that a maximal admissible subalgebra of  $\mathcal{A}_6$ is
		$Fe_{22}\oplus F(e_{33}+e_{44})\oplus F(e_{34}+e_{43})$.
		Moreover, $Z(\mathcal{A}_6)=G^{(0)}e_{15}$,
		$[x_1,x_2][x_3,x_4][x_5,x_6,x_7]$ is a proper central polynomial and, so, $exp(\mathcal{A}_6)= exp^\delta(\mathcal{A}_6)=3$.

		A maximal admissible subalgebra of the algebra $\mathcal{A}_7$ is
		$F(e_{22}+e_{33})\oplus F(e_{23}+e_{32})\oplus Fe_{44}.$
		Moreover $Z(\mathcal A_{7})=G^{(0)}e_{15}$  and
		$[x_1,x_2,x_3][x_4,x_5][x_6,x_7]$ is a proper central polynomial.
		Thus $ exp(\mathcal{A}_7) =exp^\delta(\mathcal{A}_7)=3$.

		Finally, a maximal admissible subalgebra for both algebras $\mathcal{A}_8$ and $\mathcal{A}_9$ is
		$F(e_{11}+e_{44})\oplus F(e_{22}+e_{33})\oplus F(e_{23}+e_{32})$. Hence $exp(\mathcal{A}_8)=exp(\mathcal{A}_9)=3.$
		The center of $\mathcal{A}_8$ is $Z(\mathcal{A}_8)=G^{(0)}(e_{11}+e_{22}+e_{33}+e_{44})+G^{(0)}e_{14}$ and the center of $\mathcal{A}_9$ is
		$Z(\mathcal{A}_9)=G^{(0)}(e_{11}+e_{22}+e_{33}+e_{44})+G^{(1)}e_{14}$.
		Moreover $[x_1,x_2,x_3][x_4,x_5,x_6]$ is a proper central polynomial for both $\mathcal{A}_8$ and $\mathcal{A}_9$. Hence
		$exp^\delta(\mathcal{A}_8)=exp^\delta(\mathcal{A}_9)=3$.
	\end{proof}
	
	Next we will determine generators for the $T$-ideal of identities of the algebras  $\mathcal{A}_6$ and  $\mathcal{A}_7$.
	To this end we first need to prove a preliminary lemma. Let $\langle f_1,\ldots,f_r\rangle_T$ denote the $T$-ideal of the free algebra generated by the polynomials $f_1,\ldots,f_r.$ We have

	\begin{Lemma} \label{TidealB}
		\hfill
		\begin{itemize}
			\item[1)]
			If
			$
			C_1 = \{ (a_{ij})\in M_3(G) \mid a_{11}=a_{22}, a_{13} \in G^{(0)},  a_{12}=a_{21}, a_{23}\in G^{(1)},  a_{31}=a_{32}=a_{33}=0\}
			$
			then
			$ Id(C_1)=\langle [x_1,x_2,x_3]x_4\rangle_T$.
			\item[2)]
			If
			$
			C_2 = \{ (a_{ij})\in M_3(G) \mid a_{22}=a_{33}, a_{12} \in G^{(0)},  a_{23}=a_{32}, a_{13}\in G^{(1)},
			a_{11}=a_{21}=a_{31}=0\}
			$
			then
			$Id(C_2)=\langle x_1 [x_2,x_3,x_4]\rangle_T$.
		\end{itemize}
	\end{Lemma}
	
	\begin{proof}
		It is easily checked that $C_1$ satisfies the identity $[x_1,x_2,x_3]x_4\equiv 0$.
		Hence Id$(C_1)\supseteq \langle [x_1,x_2,x_3]x_4\rangle_T$.
		
		Next we shall work modulo the identity $[x_1,x_2,x_3]x_4\equiv 0$.
		First we claim that
		\begin{equation}\label{commutatore}
			[x_1,x_2]x_3x_4\equiv x_3 [x_1,x_2]x_4.
		\end{equation}
		Indeed
		$$
		[x_1,x_2]x_3x_4=([[x_1,x_2], x_3]+ x_3 [x_1,x_2])x_4= [x_1,x_2,x_3]x_4+ x_3 [x_1,x_2]x_4\equiv x_3 [x_1,x_2]x_4.
		$$
		
		\noindent
		Moreover
		from the equality
		$$
		0\equiv[x_1,x_2x_3,x_4]x_5=[[x_1,x_2]x_3+x_2[x_1,x_3], x_4]x_5=
		$$
		$$
		([x_1,x_2,x_4]x_3+[x_1,x_2][x_3,x_4]+[x_2,x_4][x_1,x_3]+x_2[x_1,x_3,x_4])x_5
		$$
		we get
		$$
		[x_1,x_2][x_3,x_4]x_5+[x_2,x_4][x_1,x_3]x_5\equiv 0.
		$$
		Since by (\ref{commutatore}),  $[x_2,x_4][x_1,x_3]x_5\equiv [x_1,x_3][x_2,x_4]x_5$, we obtain
		\begin{equation}\label{commutatore2}
			[x_1,x_2][x_3,x_4]x_5+[x_1,x_3][x_2,x_4]x_5\equiv 0.
		\end{equation}
		From the relations (\ref{commutatore}) and (\ref{commutatore2}) it follows that any monomial
		$$
		w=x_{u_1}\cdots x_{u_{n-1}}x_a \in P_n
		$$
		can be written modulo the identity $[x_1,x_2,x_3]x_4$ as a linear combination of products of the type
		$$
		x_{i_1}\cdots x_{i_{k}}[x_{j_1},x_{j_2}]\cdots [x_{j_{2m-1}},x_{j_{2m}}] x_a
		$$
		with $i_1<\cdots < i_k$, $j_1<\cdots < j_{2m}$, $k+2m=n-1$.
		
		Let $f(x_1, \ldots , x_n)$ be a multilinear polynomial identity of $C_1$, we shall prove that it vanishes modulo the identity $[x_1,x_2,x_3]x_4$.
		Suppose that $f(x_1, \ldots , x_n)\ne 0$ modulo $[x_1,x_2,x_3]x_4$ and write
		$$
		f(x_1, \ldots , x_n)\equiv\sum \beta_{\underline i,\underline j, a}x_{i_1}\cdots x_{i_{n-2m-1}}[x_{j_1},x_{j_2}]\cdots [x_{j_{2m-1}},x_{j_{2m}}] x_a,
		$$
		where $\beta_{\underline i,\underline j, a} \in F$, $\underline i =(i_1, \ldots, i_{n-2m-1})$,
		$i_1<\cdots < i_{n-2m-1}$, $\underline j=(j_1, \ldots,  j_{2m})$ and
		$j_1<\cdots < j_{2m}$.
		We choose the minimal $m$ such that for some  $(n-2m-1, 2m)$-tuple $(i_1,\ldots
		,i_{n-2m-1}; j_1, \ldots, j_{2m})$ we have that $\beta_{\underline i,\underline j, a} \neq 0$.
		We make the substitution $x_{i_1}=\cdots = x_{i_{n-2m-1}}= e_{11}+e_{22}$, $x_{j_t}=g_t(e_{12}+e_{21})$ for all $t=1, \ldots ,  2m$,  $x_a= g e_{13}$, where
		$g \in G^{(0)}$,  the $g_t$'s are distinct generators of $G^{(1)}$   and $gg_t\neq 0$.
		Clearly we obtain, for a suitable coefficient $\beta_{\underline i,\underline j, a} \in F$,
		$$
		\beta_{\underline i,\underline j, a} g_1\cdots g_{2m}ge_{13}\neq 0.
		$$
		
		\noindent
		The minimality of $m$ gives that all other coefficients vanish and we get a contradiction.
		
		Similar considerations show that Id$(C_2)=\langle x_1 [x_2,x_3,x_4]\rangle_T$.
	\end{proof}
	
	\bigskip
	
	Recall that the free supercommutative algebra $S = F[U, V]$ is the algebra with $1$ generated by two countable sets
	$U =\{u_1, u_2 , \ldots\}$ and $V =\{v_1, v_2 , \ldots\}$ over $F$, subject to the condition that the elements of $U$ are central and the elements of $V$
	anticommute among them.
	The algebra $S$ has a natural $\mathbb Z_2$-grading  $S=S^{(0)} \oplus S^{(1)}$ if we require the variables of $U$ to be even and those of $V$ to be odd. We shall
	call the elements of $V$ Grassmann variables.
	
	If $B=B^{(0)}\oplus B^{(1)}$ is a superalgebra, the algebra $S(B) = (S^{(0)} \otimes B^{(0)}) + (S^{(1)} \otimes B^{(1)})$ is called  the superenvelope of $B.$
	It has the following basic property (see \cite[Proposition 3.8.5]{GZbook}).
	
	\begin{Proposition} \label{supercomm}
		Let $B=B^{(0)}\oplus B^{(1)}$ be a finite-dimensional superalgebra.
		If $\{a_1, \ldots, a_k \} $ and $ \{b_1, \ldots, b_t \} $ are basis of $B^{(0)}$ and $B^{(1)}$ respectively, then the subalgebra of $S(B)$ generated by the
		elements
		$$
		\xi_i = u_{i_1} \otimes a_1 +\cdots + u_{i_k}  \otimes a_k+ v_{i_1}  \otimes b_1 +\cdots + v_{i_t} \otimes b_t, \,\, i= 1,2, \ldots
		$$
		is a relatively free algebra of the variety $var(G(B))=var(S(B))$ with free generators $\xi_1$, $\xi_2, \ldots.$
	\end{Proposition}
	
	Now, we can prove the following
	\begin{Proposition} \label{78}
		\hfill
		\begin{itemize}
			\item[1)] Id$(\mathcal{A}_6)=\langle x_1[x_2,x_3][x_4,x_5,x_6]x_7\rangle
			_T$.
			\item[2)] Id$(\mathcal{A}_7)=\langle x_1[x_2,x_3,x_4][x_5,x_6]x_7\rangle
			_T$.
		\end{itemize}
	\end{Proposition}
	
	\begin{proof}
		First we consider the algebra $\mathcal{A}_6$.
		We notice that $\mathcal{A}_6=G(B)$ where $B=B^{(0)}\oplus B^{(1)}$ is a finite-dimensional superalgebra with dim$_FB^{(0)}=8$ and dim$_FB^{(1)}=4$.
		We write
		
		$$
		\mathcal{A}_6=\left(
		\begin{array}{cc}
			B_1 & J \\
			0 & B_2 \\
		\end{array}
		\right)
		$$
		
		\medskip
		\noindent
		where
		$B_1=\left(
		\begin{array}{cc}
			0 & G^{(0)} \\
			0 & G^{(0)}  \\
		\end{array}
		\right)$,
		$B_2=\left\{\left(
		\begin{array}{ccc}
			g_1 & h_1 & G^{(0)} \\
			h_1 & g_1 & G^{(1)} \\
			0 & 0  & 0 \\
		\end{array}
		\right) \mid g_1 \in G^{(0)}, h_1\in G^{(1)} \right\}$  and
		$J=\left(
		\begin{array}{ccc}
			G^{(0)} & G^{(1)} & G^{(0)} \\
			G^{(0)} & G^{(1)} & G^{(0)} \\
		\end{array}
		\right)$.

		\medskip
		
		We shall prove that Id$(\mathcal{A}_6)=$Id$(B_1)$Id$(B_2)$ by making use of Proposition \ref{supercomm}.
		
		We choose $8n$ polynomial variables $\xi^i_{j}$, $1\le i\le n$, $i\le j\le 8$, and $4n$ Grassmann variables  $\eta^i_{l}$, $1\le i\le n$, $1\le l\le 4$.
		Denote by $\mathcal{G}$ the subalgebra generated by the elements
		$$
		Z^i=\sum_{j=1}^8 \xi^i_{j} a_j + \sum_{l=1}^4 \eta^i_{l} b_l \in  F[\xi^i_{j}, \eta^i_{l}]\otimes B,
		$$
		where $a_1, \ldots , a_8$ are a basis of $B^{(0)}$ and $b_1, \ldots , b_4$ a basis of $B^{(1)}$.
		Since we choose a basis of $B$ made of matrix units, we shall use a double index for the commutative and Grassmann variables.
		
		Consider three infinite disjoint sets of commutative and Grassmann variables and denote them as follows:
		$$
		X=\{\xi^i_{1,2}, \xi^i_{2,2}\}, \,\, Y=\{\xi^i_{3,3}, \xi^i_{3,5}, \eta^i_{3,4}, \eta^i_{4,5}\}
		$$
		and
		$$
		U=\{\xi^i_{1,3}, \xi^i_{1,5}, \xi^i_{2,3}, \xi^i_{2,5}, \eta^i_{1,4}, \eta^i_{2,4}\}.
		$$
		For $i=1, \ldots , n$,  consider the following matrices:
		$$
		X^i=\xi^i_{1,2}e_{12}+ \xi^i_{2,2}e_{22},
		$$
		$$
		Y^i=\xi^i_{3,3}(e_{33}+e_{44})+ \xi^i_{3,5}e_{35}+\eta^i_{3,4}(e_{34}+e_{43})+\eta^i_{4,5}e_{45},
		$$
		$$
		U^i=\xi^i_{1,3}e_{13}+ \xi^i_{1,5}e_{15}+\xi^i_{2,3}e_{23}+\xi^i_{2,5}e_{25}+ \eta^i_{1,4}e_{14}
		+ \eta^i_{2,4}e_{24}.
		$$
		Then, by Proposition \ref{supercomm}, the generic matrices $X^1, \ldots , X^n$ generate the relatively free algebra $\Tilde{B}_1$ of rank $n$ of
		the variety $var(B_1)$ and $Y^1, \ldots , Y^n$ generate the relatively free algebra $\Tilde{B}_2$ of rank $n$ of
		the variety $var(B_2)$.
		
		If we prove that $U^1, \ldots , U^n$ generate a free
		$(\Tilde{B}_1, \Tilde{B}_2$)-bimodule then, by Lewin theorem \cite[Corollary 1.8.2]{GZbook},
		we obtain that the set  $\{ Z^i=X^i+ Y^i+ U^i \mid 1\le i\le n\}$ generates a relatively free
		algebra of the variety corresponding to  the product of the $T$-ideals Id$(B_1)$Id$(B_2)$.
		This will complete the proof of the proposition.
		
		Hence we need to show that $U^1, \ldots , U^n$ generate a free
		$(\Tilde{B}_1, \Tilde{B}_2)$-bimodule.
		Let $a^j_1, a^j_2, \ldots $, $1\le j\le n$,  be  non-zero elements  of $\Tilde{B}_1$
		and let  $b_1,b_2, \ldots $  be linearly independent elements of $\Tilde{B}_2$.
		We want to prove that
		$$
		\sum_ia^1_iU^1b_i+\cdots +a^n_iU^nb_i\neq 0.
		$$
		Since $U^1, \ldots , U^n$ depend on disjoint sets of variables it is enough to check that
		$$
		a^1_1U^1b_1+\cdots +a^1_tU^1b_t\neq 0,
		$$
		for all $t\geq 1$.
		Since $a_1^1\neq 0$, it has a non zero value in $B_1$ under some specialization of the variables in $X$. This value is of the type
		$\lambda_{1,2}^1e_{12}+\lambda_{2,2}^1e_{22}\neq 0,$
		with $\lambda_{1,2}^1, \lambda_{2,2}^1 \in G^{(0)}$.
		Suppose now that
		$
		a_1^1U^1b_1+\cdots +a_t^1U^1b_t= 0.
		$
		Then
		$$
		\sum_{i=1}^t a_i^1(\xi^i_{2,3}e_{23}+ \xi^i_{2,5}e_{25}+ \eta^i_{2,4}e_{24}) b_i=0
		$$
		and, so,
		$$
		\sum_{i=1}^t (\lambda_{1,2}^ie_{12}+\lambda_{2,2}^ie_{22})(\xi^i_{2,3}e_{23}+ \xi^i_{2,5}e_{25}+ \eta^i_{2,4}e_{24}) b_i =0.
		$$
		Assume for instance that $\lambda_{1,2}^1\neq 0$.
		If we consider the following specialization
		$
		\xi^i_{2,3}\rightarrow 1, \,\, \xi^i_{2,5}\rightarrow 0, \,\, \eta^i_{2,4}\rightarrow 0,
		$
		then
		$$
		\sum_{i=1}^t (\lambda_{1,2}^ie_{13}+\lambda_{2,2}^ie_{23})b_i =0
		$$
		and in particular
		$$
		e_{13}(\sum_{i=1}^t \lambda_{1,2}^i b_i) =0.
		$$
		Let $f=\sum_{i=1}^t \lambda_{1,2}^i b_i.$ This element is not an identity for $B_2$ but for any specialization of the variables in $Y$,  i.e. for any
		homomorphism $\phi: \Tilde{B}_2\rightarrow B_2$, we have $e_{13}\phi(f)=0.$ This means that in $\phi(f)$
		all elements of the first row are zero. By  \cite[Lemma 4]{GZ2003AM}, the linear span $\langle\phi(f) \rangle$ of all values of $f$ on $B_2$ is a Lie ideal of
		$B_2$ and this is false.
		The contradiction just obtained proves that Id$(\mathcal{A}_6)=$Id$(B_1)$Id$(B_2)$. Since, by Lemma \ref{TidealB}, Id$(B_1)=\langle
		x_1[x_2,x_3]\rangle_T$ and Id$(B_2)=\langle[x_4, x_5,x_6]x_7\rangle_T$,
		we obtain that Id$(\mathcal{A}_6)=\langle
		x_1[x_2,x_3][x_4, x_5,x_6]x_7\rangle_T$.

		The proof of the second part of the proposition is very similar to the above proof and we omit it.
	\end{proof}

	\begin{Remark} \label{rem1}
		Consider the following algebras
		\begin{itemize}
			\item
			$\mathcal{A}_6^1=
			\{ (a_{ij}) \in G(UT(2,3)_{(0,0,0,1,1)}) \mid a_{11}=a_{21}=a_{53}=a_{54}=a_{55}=0,\ a_{33}= a_{44}, \ a_{34}=a_{43} \}$,
			\medskip
			
			\item
			$\mathcal{A}_6^2= \{ (a_{ij}) \in G(UT(2,3)_{(0,1,1,0,1)})\mid  a_{11}=a_{21}=a_{53}=a_{54}=a_{55}=0,\ a_{33}=a_{44}, \ a_{34}=a_{43} \}$,
			\medskip
			\item
			$\mathcal{A}_6^3= \{(a_{ij})\in G(UT(2,3)_{(0,1,1,0,0)})\mid  a_{11}=a_{21}=a_{53}=a_{54}=a_{55}=0, \ a_{33}=a_{44},\
			a_{34}=a_{43}
			\}$,
			
			\medskip
			
			\item
			$\mathcal{A}_7^1= \{ (a_{ij}) \in G(UT(3,2)_{(0,1,0,1,1)}) \mid a_{11}=a_{21}=a_{31}=a_{54}=a_{55}=0 ,\  a_{22}=a_{33},\
			a_{23}=a_{32} \}$,
			\medskip
			
			\item
			
			$\mathcal{A}_7^2= \{ (a_{ij}) \in G(UT(3,2)_{(0,0,1,0,1)}) \mid a_{11}=a_{21}=a_{31}=a_{54}=a_{55}=0,\ a_{22}=a_{33},\
			a_{23}=a_{32}
			\} $,
			\medskip
			
			\item
			$\mathcal{A}_7^3= \{(a_{ij}) \in G(UT(3,2)_{(0,1,0,1,0)}) \mid a_{11}=a_{21}=a_{31}=a_{54}=a_{55}=0,\ a_{22}=a_{33},\
			a_{23}=a_{32}\} $.
			\medskip
		\end{itemize}
		By  Proposition \ref{78} and its proof it follows that
		$$
		Id(\mathcal{A}_6^1)=Id(\mathcal{A}_6^2)=Id(\mathcal{A}_6^3)=Id(\mathcal{A}_6)=\langle x_1[x_2,x_3][x_4,x_5,x_6]x_7\rangle_T
		$$
		and
		$$
		Id(\mathcal{A}_7^1)=Id(\mathcal{A}_7^2)=Id(\mathcal{A}_7^3)=Id(\mathcal{A}_7)= \langle x_1[x_2,x_3,x_4][x_5,x_6]x_7\rangle_T.
		$$
	\end{Remark}

	\bigskip
	
	\section{Building up algebras of small $\delta$-exponent}

	Throughout this section  $B=B^{(0)} \oplus B^{(1)}$ will be a finite dimensional superalgebra  over an algebraically closed field $F$ of characteristic zero.
	
	We start with the following.
	
	\begin{Lemma} \label{lemmaA_3}
		If $B^{(0)}$ contains three orthogonal idempotents $e_1,e_2,e_3$ such that
		$e_1j_1e_2j_2e_3j_3e_1\ne 0$, for some $j_1,j_2,j_3\in J(B)$, then there exists a $\mathbb{Z}_2$-graded subalgebra $\bar B$ of $B$  such that $Id(G(\bar
		B))\subseteq Id(\mathcal{A}_3)$.
	\end{Lemma}
	
	\begin{proof}
		By linearity we assume, as we may, that the elements $j_1, j_2,j_3$ are homogeneous.
		Let $\bar B$ be the $\mathbb{Z}_2$-graded subalgebra of $B$ generated by the homogeneous elements
		$$
		e_1,  e_2,  e_3, e_1j_1e_2, e_2j_2e_3, e_3j_3e_1.
		$$
		We build a homomorphism of superalgebras $\varphi :\bar B\rightarrow UT_4(F)$  by setting $\varphi(e_1)=e_{11}+e_{44}$, $\varphi(e_2)=e_{22}$,
		$\varphi(e_3)=e_{33}$, $\varphi(e_1j_1e_2)=e_{12}$, $\varphi(e_2j_2e_3)=e_{23}$ and $\varphi(e_3j_3e_1)=e_{34}$. Then $I=Ker\varphi$ is the ideal of $\bar B$
		generated by the element
		$$
		e_3j_3e_1j_1e_2
		$$
		and $\bar B/I\simeq \varphi(\bar B)$ is isomorphic to the algebra $\mathcal{A}_3$ with a suitable $\mathbb{Z}_2$-grading.
		
		Now, it can be checked that the elements
		$$
		e_1,  e_2,  e_3, e_1j_1e_2, e_2j_2e_3, e_3j_3e_1, e_1j_1e_2j_2e_3, e_2j_2e_3j_3e_1, e_1j_1e_2j_2e_3j_3e_1
		$$
		are linearly independent and form a basis of $\bar B$ mod $I$.
		By abuse of notation we identify these representatives with the corresponding cosets.
		
		Let $g_1,g_2,g_3$ be distinct homogeneous elements of $G$ having the same homogeneous degree as  $j_1,j_2,j_3$, respectively. Then the elements
		$$
		1\otimes e_1,\  1\otimes e_2,\  1\otimes e_3,\ g_1\otimes e_1j_1e_2,\ g_2\otimes e_2j_2e_3,\ g_3\otimes e_3j_3e_1,
		$$
		$$
		g_1g_2\otimes e_1j_1e_2j_2e_3,\ g_2g_1\otimes e_2j_2e_3j_3e_1,\  g_1g_2g_3 \otimes e_1j_1e_2j_2e_3j_3e_1,
		$$
		form a basis of a subalgebra $\mathcal C$ of $G(\bar B/I)$ isomorphic to $\mathcal{A}_3$ with trivial grading.
		
		It follows that  $Id(\mathcal{A}_3)= Id(\mathcal C) \supseteq Id(G(\bar B/I)) \supseteq Id(G(\bar B))$, and the proof is complete.
	\end{proof}

	\begin{Lemma} \label{lemmaA_4}
		If $B^{(0)}$ contains three orthogonal idempotents $e_1,e_2,e_3$ such that
		$j_1e_1j_2e_2j_3e_3j_4\ne 0$ for some  $j_1,j_2,j_3, j_4\in J(B)$, then there exists a $\mathbb{Z}_2$-graded subalgebra $\bar B$ of $B$  such that $Id(G(\bar
		B))\subseteq Id(\mathcal{A}_4)$.
	\end{Lemma}
	
	\begin{proof}
		As in the previous lemma we assume that the elements  $j_1,j_2,j_3, j_4$ are homogeneous.
		Let $\bar B$ be the $\mathbb{Z}_2$-graded subalgebra of $B$ generated by the homogeneous elements
		$$
		e_1,  e_2,  e_3, j_1e_1, e_1j_2e_2, e_2j_3e_3, e_3j_4.
		$$
		We build a homomorphism of superalgebras $\psi :\bar B\rightarrow UT_5(F) $ defined by setting $\psi(e_1)=e_{22}$, $\psi(e_2)=e_{33}$, $\psi(e_3)=e_{44}$,
		$\psi(j_1e_1)=e_{12}$, $\psi(e_1j_2e_2)=e_{23}$, $\psi(e_2j_3e_3)=e_{34}$ and $\psi(e_3j_4)=e_{45}$. Then $I=Ker \psi$ is the homogeneous ideal of $\bar B$
		generated by the elements
		$$
		e_3j_4e_1, \, e_3j_4e_2, \, e_3j_4e_3, \, e_3j_4j_1e_1, \, e_3j_1e_1, \,  e_2j_1e_1, \,
		e_1j_1e_1
		$$
		and $\bar B/I\simeq \psi(\bar B)$ is isomorphic to the algebra $\mathcal{A}_4$ with a suitable $\mathbb{Z}_2$-grading.
		
		The elements
		$$
		e_1,  e_2,  e_3, \
		j_1e_1,\
		e_1j_2e_2,\
		e_2j_3e_3,\
		e_3j_4,\ j_1e_1j_2e_2,\ e_1j_2e_2j_3e_3,\
		$$
		$$
		e_2j_3e_3j_4,\ j_1e_1j_2e_2j_3e_3,\
		e_1j_2e_2j_3e_3j_4,\
		j_1e_1j_2e_2j_3e_3j_4
		$$
		are linearly independent  and form a basis of $\bar B$ mod $I$.
		
		By abuse of notation we identify these representatives with the corresponding cosets.
		It follows that, for distinct $g_i\in  G^{(0)}\cup G^{(1)}$ having the same homogeneous degree of the corresponding $j_i$, the elements
		$$
		1\otimes e_1,  1\otimes e_2,  1\otimes e_3,\ g_1\otimes  j_1e_1,\
		g_2 \otimes e_1j_2e_2,\
		g_3 \otimes e_2j_3e_3,\
		g_4 \otimes e_3j_4,\
		$$
		$$
		g_1g_2  \otimes j_1e_1j_2e_2,\ g_2g_3\otimes e_1j_2e_2j_3e_3,\
		g_3g_4\otimes e_2j_3e_3j_4,\
		$$$$
		g_1g_2g_3 \otimes j_1e_1j_2e_2j_3e_3,\
		g_2g_3g_4 \otimes e_1j_2e_2j_3e_3j_4,\
		g_1g_2g_3g_4 \otimes j_1e_1j_2e_2j_3e_3j_4,
		$$
		
		\smallskip
		\noindent are a basis of a subalgebra
		$\mathcal C$ of $G(\bar B/I)$ isomorphic to $\mathcal{A}_4$ with trivial grading.
		
		We have that $Id(\mathcal{A}_4)= Id(\mathcal C) \supseteq Id(G(\bar B/I)) \supseteq Id(G(\bar B))$, a desired conclusion.
	\end{proof}

	\begin{Lemma} \label{lemmaA_5}
		If $B^{(0)}$ contains two orthogonal idempotents $e_1,e_2$, with $e_2\in F\oplus cF$, such that
		$e_2j_2e_1j_1e_2\ne 0$, for some $j_1,j_2\in J(B)$,   then there exists a $\mathbb{Z}_2$-graded subalgebra $\bar B$ of $B$  such that $Id(G(\bar B))\subseteq
		Id(\mathcal{A}_5)$.
	\end{Lemma}
	\begin{proof}
		Let $\bar B$ be the $\mathbb{Z}_2$-graded subalgebra of $B$ generated by the homogeneous elements
		$$
		e_1,  e_2, ce_2,   e_1j_1e_2, e_2j_2e_1.
		$$ We consider the homomorphism of superalgebras $\varphi : \bar B\rightarrow UT_3(F+cF)$  by setting $\varphi(e_1)=e_{22}$, $\varphi(e_2)=e_{11}+e_{33}$,
		$\varphi(ce_2)=c(e_{11}+e_{33})$, $\varphi(e_1j_1e_2)=\alpha e_{23}$ and $\varphi(e_2j_2e_1)=\alpha e_{12}$,   where $\alpha \in \{1, c\}$ according to the
		homogeneous degree of $j_1$ and $j_2$.
		Then $I=Ker\varphi$ is the ideal of $ \bar B$ generated by
		$
		e_1j_1e_2j_2e_1, e_1j_1ce_2j_2e_1
		$
		and $ \bar B/I\simeq \varphi( \bar B)=\bar A$ where
		$$\bar A=\left \{\begin{pmatrix}
			a & x & y  \\
			0& b & z  \\
			0 & 0 & a
		\end{pmatrix} : b\in F, a,x,y,z\in F+cF \right \}.
		$$
		\noindent
		The following elements
		$$
		e_1,  e_2, ce_2,   e_1j_1e_2, e_2j_2e_1, e_1j_1ce_2, ce_2j_2e_1, e_2j_2e_1j_1e_2, ce_2j_2e_1j_1e_2
		$$
		are linearly independent and form a basis of $\bar B$ mod $I.$  We identify these representatives with the corresponding cosets.
		Then the elements
		$$
		1\otimes e_1,  1\otimes e_2,  1\otimes ce_2, 1 \otimes e_1j_1e_2, 1 \otimes  e_2j_2e_1,$$ $$ 1 \otimes e_1j_1ce_2,  1 \otimes ce_2j_2e_1,  1 \otimes
		e_2j_2e_1j_1e_2,  1 \otimes ce_2j_2e_1j_1e_2
		$$
		form a basis of $G(\bar B/I)\simeq G(\bar A)=\mathcal{A}_5$ and
		$Id(\mathcal{A}_5) \supseteq Id(G(\bar B))$.
	\end{proof}

	\begin{Lemma} \label{lemmaA_6}
		If $B^{(0)}$ contains two orthogonal idempotents $e_1,e_2$, with $e_2\in F\oplus cF$, such that
		$j_1e_1j_2e_2j_3\ne 0$ for some  $j_1,j_2,j_3\in J(B)$, then there exists a $\mathbb{Z}_2$-graded subalgebra $\bar B$ of $B$  such that $Id(G(\bar B))\subseteq
		Id(\mathcal{A}_6)$ .
	\end{Lemma}
	\begin{proof}
		Let us assume that the elements  $j_1,j_2, j_3$ are homogeneous of degree zero and let $\bar B$ be the $\mathbb{Z}_2$-graded subalgebra of $B$ generated by the
		homogeneous elements
		$$
		e_1,  e_2, ce_2,   j_1e_1, e_1j_2e_2, e_2j_3.
		$$
		such that $j_1e_1j_2e_2j_3\ne 0$.
		We consider the homomorphism of superalgebras $\varphi : \bar B\rightarrow UT(2,3)_{(0,0,0,1,0)}$  by setting $\varphi(e_1)=e_{22}$,
		$\varphi(e_2)=e_{33}+e_{44}$, $\varphi(ce_2)=e_{34}+e_{43}$, $\varphi(j_1e_1)=e_{12}$, $\varphi(e_1j_2e_2)=e_{23}$   and $\varphi(e_2j_3)=e_{35}$.
		Let $I=Ker\varphi$ then $\bar B/I\simeq \varphi(\bar B)=\bar A$ where

		$$
		\bar A= \begin{pmatrix}
			0 & a & b & d & e  \\
			0 & f & g & h & l \\
			0 & 0 & u & v & m \\
			0 & 0 & v & u & n \\
			0 & 0 & 0 & 0& 0
		\end{pmatrix} =  \begin{pmatrix}
			0 & a & b & 0 & e  \\
			0 & f & g & 0 & l \\
			0 & 0 & u & 0 & m \\
			0 & 0 & 0 & u & 0 \\
			0 & 0 & 0 & 0& 0
		\end{pmatrix}  \oplus \begin{pmatrix}
			0 & 0 & 0 & d & 0  \\
			0 & 0 & 0 & h & 0 \\
			0 & 0 & 0 & v & 0 \\
			0 & 0 & v & 0 & n \\
			0 & 0 & 0 & 0& 0
		\end{pmatrix}
		$$
		
		\smallskip
		
		The following elements
		$$
		e_1,  e_2,\ ce_2,\   j_1e_1,\ e_1j_2e_2,\ e_2j_3,\ j_1e_1j_2e_2,
		$$
		$$
		j_1e_1j_2ce_2,\ j_1e_1j_2e_2j_3,\ e_1j_2ce_2,\ e_1j_2e_2j_3,\
		ce_2j_3
		$$
		are linearly independent mod $I$ and form a basis of $\bar B/I$.
		By the same procedure of the previous lemmas it follows that $G(\bar A)=\mathcal{A}_6$
		and so  $Id(\mathcal{A}_6) \supseteq Id(G(\bar B))$.
		
		\smallskip
		The same result is obtained when $j_1 \in J^{(0)}$ and $j_2, j_3 \in J^{(1)}.$
		Moreover if $j_1 \in J^{(0)}$ and $j_2, j_3$ are of different homogeneous degree
		we obtain that  $Id(\mathcal{A}_6^1) \supseteq Id(G(\bar B))$.

		In case $j_1 \in J^{(1)}$ we get that either $Id(\mathcal{A}_6^2) \supseteq Id(G(\bar B))$
		or $Id(\mathcal{A}_6^3) \supseteq Id(G(\bar B))$ according to whether $j_2, j_3$ are of the same or of different homogeneous degree. Since, by Remark
		\ref{rem1}, $Id(\mathcal{A}_6^1)=Id(\mathcal{A}_6^2)=Id(\mathcal{A}_6^3)=Id(\mathcal{A}_6)$, we get the desired result.
	\end{proof}
	
	\begin{Lemma} \label{lemmaA_7}
		If $B^{(0)}$ contains two orthogonal idempotents $e_1,e_2$, with $e_2\in F\oplus cF$, such that
		$j_1e_2j_2e_1j_3\ne 0$ for some  $j_1,j_2,j_3\in J(B)$, then there exists a $\mathbb{Z}_2$-graded subalgebra $\bar B$ of $B$  such that $Id(G(\bar B))\subseteq
		Id(\mathcal{A}_7)$.
	\end{Lemma}
	\begin{proof}
		By following the proof of Lemma \ref{lemmaA_6} we obtain that,
		if $j_3\in J^{(0)}$ and $j_1, j_2$ are of the same homogeneous degree
		then $Id(\mathcal{A}_7) \supseteq Id(G(\bar B))$. Otherwise, if $j_3\in J^{(0)}$ and $j_1, j_2$ are of different homogeneous degree
		then $Id(\mathcal{A}_7^1) \supseteq Id(G(\bar B))$. Similarly, if $j_3\in J^{(1)},$ we get either $Id(\mathcal{A}_7^2) \supseteq Id(G(\bar B))$ or
		$Id(\mathcal{A}_7^3) \supseteq Id(G(\bar B))$.
		
		Recalling  Remark \ref{rem1} we have that  $Id(\mathcal{A}_7^1)=Id(\mathcal{A}_7^2)=Id(\mathcal{A}_7^3)=Id(\mathcal{A}_7)$ and we are done.
	\end{proof}
	
	\begin{Lemma} \label{lemmaA_8}
		If $B^{(0)}$ contains two orthogonal idempotents $e_1,e_2$, with $e_2\in F\oplus cF$, such that
		$e_1j_1e_2j_2e_1\ne 0$, for some $j_1,j_2\in J(B)$,   then there exists a $\mathbb{Z}_2$-graded subalgebra $\bar B$ of $B$  such that $Id(G(\bar B))\subseteq
		Id( \mathcal{A}_8)$ or $Id(G(\bar B))\subseteq Id( \mathcal{A}_9)$.
	\end{Lemma}
	
	\begin{proof}
		First we suppose that the elements  $j_1,j_2$ are homogeneous of degree zero.
		Let $\bar B$ be the $\mathbb{Z}_2$-graded subalgebra of $B$ generated by the homogeneous elements
		$$
		e_1,  e_2,\ ce_2,\   e_1j_1e_2,\ e_2j_2e_1.
		$$
		We consider the homomorphism of superalgebras $\varphi : \bar B\rightarrow UT(1,2,1)_{(0,0,1,0)}$  by setting $\varphi(e_1)=e_{11}+e_{44}$,
		$\varphi(e_2)=e_{22}+e_{33}$, $\varphi(ce_2)=e_{23}+e_{32}$ and $\varphi(e_1j_1e_2)=e_{12}$, $\varphi(e_2j_2e_1)=e_{24}$.
		Then $I=Ker\varphi$ is the ideal of $\bar B$ generated by
		$
		e_2j_2e_1j_1e_2, e_1j_1ce_2j_2e_1
		$
		and $\bar B/I\simeq \varphi(\bar B)= \bar A$ where

		$$
		\bar A= \begin{pmatrix}
			a & b & d & e  \\
			0 & u & v  & h \\
			0 & v & u & m \\
			0 & 0 &  0  & a
		\end{pmatrix} =  \begin{pmatrix}
			a & b & 0 & e \\
			0  & u & 0  & h \\
			0 & 0 & u & 0\\
			0 &  0 &  0 & a
		\end{pmatrix}  \oplus \begin{pmatrix}
			0 & 0 & d & 0  \\
			0& 0 & v  & 0 \\
			0 & v & 0 & m \\
			0 &  0 & 0  & 0
		\end{pmatrix}
		$$
		
		\smallskip
		
		The following elements
		$$
		e_1,  e_2,\ ce_2,\   e_1j_1e_2,\ e_2j_2e_1,\ e_1j_1ce_2,\ ce_2j_2e_1,\ e_1j_1e_2j_2e_1
		$$
		are linearly independent and form a basis of $\bar B$ mod $I.$  We identify these representatives with the corresponding cosets.
		The elements
		$$
		1\otimes e_1,  1\otimes e_2,  1\otimes ce_2, 1 \otimes e_1j_1e_2, 1 \otimes  e_2j_2e_1, 1 \otimes e_1j_1ce_2,  1 \otimes ce_2j_2e_1,  1 \otimes e_1j_1e_2j_2e_1
		$$
		form a basis of $G(\bar B/I)$ isomorphic to $\mathcal{A}_8$.
		Thus  $Id( \mathcal{A}_8) \supseteq Id(G(\bar B))$.
		
		If $j_1, j_2 \in J^{(1)}$ we obtain the same result.
		If $j_1$ and $j_2$ are of different homogeneous degree we get a homomorphism of superalgebras $\varphi : \bar B\rightarrow UT(1,2,1)_{(0,0,1,1)}$ such that
		$\bar B/I\simeq \varphi(\bar B)= \bar A$ where
		
		$$\bar A=\begin{pmatrix}
			a & b & d & e  \\
			0& u & v  & h \\
			0 & v & u & m \\
			0 &  0 & 0  & a
		\end{pmatrix} =  \begin{pmatrix}
			a & b & 0 & 0  \\
			0& u & 0  & 0 \\
			0& 0 & u & m\\
			0 &0   & 0  & a
		\end{pmatrix}  \oplus \begin{pmatrix}
			0 & 0 & d & e  \\
			0 & 0 & v  & h \\
			0 & v & 0 & 0 \\
			0 &  0 &  0 & 0
		\end{pmatrix}
		$$
		
		\smallskip
		\noindent
		Its Grassmann envelope is $G(\bar A)=\mathcal{A}_9$.
		Thus  $Id( \mathcal{A}_9) \supseteq Id(G(\bar B))$ and we are done.
	\end{proof}

	\bigskip
	
	\section{The main results}
	In this section we characterize the varieties with proper central exponent grather or equal to two.
	
	\begin{Theorem} \label{teorema1}
		Let $F$ be a  field of characteristic zero and $\mathcal{V}$ a variety of algebras over $F$. Then  $exp^\delta(\mathcal V)> 2$ if and only if
		$\mathcal{A}_i\in \mathcal{V}$ for some
		$1\le i\le 9$.
	\end{Theorem}
	
	\begin{proof}
		As we recalled in Section \ref{section2}, we may assume that  $\mathcal V = var (G(B)$) where $G(B)$ is the Grassmann envelope of a finite dimensional
		superalgebra $B=B^{(0)}+B^{(1)}$.
		Write $B=B_1\oplus \cdots \oplus B_q+J$, with the $B_i$'s simple superalgebras, and let  $\mbox{exp}^\delta (G(B))=d> 2$.
		
		If, for some $i\le q$, $B_i =M_{k,l}(F)$ then $G(B)\supseteq G(B_i)= M_{k,l}(G)$.
		If $l=0$ and $ k \ge 2$, $M_{k,0}(G)=  G^{(0)}\otimes M_k(F)\supseteq M_2(G^{(0)})$ and, so, $Id(G(B))\subseteq Id (M_2(G^{(0)})) \subseteq Id(M_2(F))= Id({\mathcal A}_1)$,
		as wished.
		If $l>0$,
		$Id(G(B))\subseteq Id(G(B_i))= Id(M_{k,l}(G))\subseteq Id(M_{1,1}(G))= Id({\mathcal A}_2)$ and we are done also in this case.
		
		Suppose now that,  for some $i\le q$, $B_i =M_{k}(F)\oplus cM_k(F)$. Then if $k\ge 2$,
		$G(B_i)=G(M_{k}(F)\oplus cM_k(F))\supseteq G^{(0)}\otimes M_k(F)\supseteq G^{(0)}\otimes M_2(F)$, and as above we get $Id(G(B))\subseteq Id(M_2(F))=
		Id({\mathcal A}_1)$.
		
		Therefore we may assume that, for  $1\le i\le q$, either $B_i=F$ with trivial grading or $B_i= F\oplus cF$.
		
		Recall that $d=\dim D$ where $D$ is a centrally admissible subalgebra of $B$ of maximal dimension.
		We may clearly assume that  $D=B_1\oplus \cdots\oplus B_r$, and, so,
		by definition there exists a proper central polynomial $f$ of $G(B)$ and an evaluation
		$0\ne f(b_1, \ldots, b_r, b_{r+1}, \ldots, b_n)\in Z(G(B))$, where $b_1\in G(B_1), \ldots, b_r\in G(B_r)$,
		$b_{r+1}, \ldots, b_n\in G(B)$.
		
		Since $exp^\delta(G(B))= \dim D=d\ge 3$ at least one of the following three cases occurs:
		\begin{itemize}
			\item[1)]
			$D$ contains three simple components of the type $F$;
			\item [2)]
			$D$ contains  one copy of $F$ and one of $F\oplus cF$;
			
			\item[3)]
			$D$ contains  two copies of $F\oplus cF$.
			
		\end{itemize}
		
		Suppose first that Case 1 occurs. $D$ contains  three orthogonal idempotents $e_1,e_2,e_3$ homogeneous of degree zero.
		Recalling that $f$ is the proper central polynomial corresponding to $D$ we have that in a non-zero evaluation $\bar f$ of $f$  at least three variables are
		evaluated in
		$$
		g_1\otimes e_1,\  g_2\otimes e_2,\  g_3\otimes e_3,
		$$
		where $g_1,g_2,g_3\in G^{(0)}$ are distinct elements.
		Since $G^{(0)}$ is central in $G$ and $f$ is multilinear, we may assume that $g_1=g_2=g_3=1$
		and, so,  the elements $1\otimes e_1,\  1\otimes e_2,\ 1\otimes e_3$ appear in the evaluation of $f$.
		
		Suppose first that a monomial of $f$ can be evaluated into a product of the type
		$$
		(1\otimes e_1)(h_1 \otimes j_1)(1\otimes e_2)(h_2 \otimes j_2) (1\otimes e_3)(h_3 \otimes j_3)(1\otimes e_1)\ne 0,
		$$
		for some $h_1\otimes j_1,h_2\otimes j_2,h_3\otimes j_3\in G^{(0)}\otimes  J^{(0)}\cup G^{(1)} \otimes  J^{(1)}.$
		Hence
		$$
		h_1h_2h_3\otimes e_1j_1e_2j_2e_3j_3e_1\ne 0
		$$
		which says that $e_1j_1e_2j_2e_3j_3e_1\ne 0$, for some homogeneous elements $j_1,j_2,j_3\in J$.
		By Lemma \ref{lemmaA_3} we get that there exists a $\mathbb{Z}_2$-graded subalgebra $\bar B$ of $B$  such that $Id(G(\bar B))\subseteq Id(\mathcal{A}_3)$ and
		we are done in this case. This clearly holds for any permutation of the $e_i$'s, $ i=1,2,3$.

		Hence, if $m$ is a monomial of $f$ whose evaluation $\bar m$ in $\bar f$ is non-zero,
		we may assume that $(1\otimes e_i)\bar m (1\otimes e_i)=0$, $1\le i\le 3$.
		This also says that
		$(1\otimes e_i)\bar f (1\otimes e_i)=0$ and, since $\bar f$ is central,  $(1\otimes e_i)\bar f=\bar f (1\otimes e_i)=0$.

		Next suppose that there exists an evaluation of a monomial of $f$ containing a product of the type
		$$
		(h_1 \otimes j_1)(1\otimes e_1)(h_2\otimes j_2) (1\otimes e_2)
		(h_3\otimes j_3)(1\otimes e_3)(h_4 \otimes j_4)\ne 0,
		$$
		for some $h_1\otimes j_1,h_2\otimes j_2,h_3\otimes j_3, h_4\otimes j_4\in G^{(0)}\otimes  J^{(0)}\cup G^{(1)} \otimes  J^{(1)}.$
		Then
		$$
		h_1h_2h_3h_4\otimes j_1e_1j_2e_2j_3e_3j_4\ne 0,
		$$
		and, so, $j_1e_1j_2e_2j_3e_3j_4\ne 0$ for some homogeneous elements $j_1,j_2,j_3, j_4\in J$.
		By Lemma \ref{lemmaA_4} we have that there exists a $\mathbb{Z}_2$-graded subalgebra $\bar B$ of $B$  such that $Id(G(\bar B))\subseteq Id(\mathcal{A}_4)$ and
		we are done.
		
		Hence we may assume that in $\bar f$, the evaluation of a monomial of $f$ never starts and ends with elements of $J$.
		As remarked above, we may also assume that  $(1\otimes e_i)\bar f=\bar f (1\otimes e_i)=0$, $1\le i\le 3$. Let us set $E_i=1\otimes e_i$, $1\le i\le 3$, and
		write
		$$
		\bar f = E_1\bar f_1+ \bar f_2 E_1+ E_2 \bar f_3 + \bar f_4 E_2+ E_3\bar f_5+ \bar f_6 E_3,
		$$
		
		\noindent where $E_1\bar f_1$  is the sum of the evaluations of the monomials of $f$ starting  with $E_1$,
		$\bar f_2 E_1$ is the  sum of the evaluations of the remaining monomials of $f$ ending with $E_1$,
		$E_2 \bar f_3 + \bar f_4 E_2$ is the sum of the evaluations of the remaining monomials starting or ending with $E_2$,
		and so on.

		Since $E_1\bar f=\bar f E_2=\bar f E_3=0$ and we may assume that
		$E_i\bar f_j E_i=0$ for $i \neq j$,  we get
		$$
		0=E_1\bar f+\bar f E_2+\bar f E_3=
		$$
		$$
		E_1\bar f_1+E_1\bar f_4E_2+E_1\bar f_6E_3+
		E_1\bar f_1 E_2+\bar f_4E_2+
		E_3\bar f_5E_2+
		E_1\bar f_1 E_3+E_2\bar f_3E_3+\bar f_6E_3.
		$$
		Moreover
		$$
		0=E_1\bar f E_2=E_1\bar f_1 E_2 + E_1\bar f_4 E_2,
		$$
		$$
		0=E_1\bar f E_3=E_1\bar f_1 E_3 + E_1\bar f_6 E_3.
		$$
		Then
		$$
		0=E_1\bar f+\bar f E_2+\bar f E_3=E_1\bar f_1 + \bar f_4 E_2
		+ \bar f_6 E_3 + E_3 \bar f_5 E_2 + E_2 \bar f_3 E_3
		$$
		and so
		$$
		E_1\bar f_1 + \bar f_4 E_2 + \bar f_6 E_3=-E_3 \bar f_5 E_2 - E_2 \bar f_3 E_3.
		$$
		
		\bigskip
		\noindent Similarly,
		since $\bar f E_1=E_2 \bar f =E_3 \bar f =0$ we get
		$$
		\bar f_2 E_1 + E_2 \bar f_3 +E_3 \bar f_5 =-E_2 \bar f_6 E_3 - E_3 \bar f_4 E_2.
		$$
		In conclusion
		$$
		\bar f= E_1 \bar f_1+ \bar f_2 E_1  +  E_2 \bar f_3+ \bar f_4 E_2+ E_3 \bar f_5 + \bar f_6 E_3
		=-E_3 \bar f_5 E_2 - E_2 \bar f_3 E_3-E_2 \bar f_6 E_3 - E_3 \bar f_4 E_2.
		$$
		Since $\bar f E_2=0$, it follows that
		$
		0=\bar f E_2=-E_3\bar f_5 E_2 - E_3 \bar f_4 E_2.
		$
		Then
		$$
		\bar f=  - E_2 \bar f_3 E_3-E_2 \bar f_6 E_3
		$$
		and
		we reach a contradiction
		$
		0=\bar f E_3=\bar f.
		$
		This completes Case 1.
		\medskip
		
		We remark that, if $D$ contains two copies of $F\oplus cF$, then it also contains one copy of $F$ and one copy of $F\oplus cF$.
		This says that Case 3 can be deduced from Case 2.
		Hence throughout we shall assume that $D$ contains  one copy of $F$ and one of $F\oplus cF$ and we only deal with Case 2.
		\medskip
		
		Hence let $D\supseteq Fe_1\oplus (Fe_2\oplus cFe_2)$.
		Then  in the non-zero evaluation $\bar f$ of the central polynomial $f$ at least one variable is evaluated in $1\otimes e_1$
		and one variable in $G^{(0)}\oplus cG^{(1)}.$ Suppose first that there exists an evaluation of a monomial of $f$ such that
		$$
		(1\otimes e_1)(h_1\otimes j_1)(g_1\otimes e_2)(h_2\otimes j_2)(1\otimes e_1) \ne 0,
		$$
		for some $h_1\otimes j_1,h_2\otimes j_2, \in G^{(0)}\otimes  J^{(0)}\cup G^{(1)} \otimes  J^{(1)}$ and $ g_1 \in G^{(0)}$.
		Then
		$$
		h_1g_1h_2\otimes e_1j_1e_2j_2e_1 \ne 0
		$$
		and by Lemma \ref{lemmaA_8} we obtain that there exists a $\mathbb{Z}_2$-graded subalgebra $\bar B$ of $B$  such that either
		$Id(G(\bar B))\subseteq Id( \mathcal{A}_8)$ or $Id(G(\bar B))\subseteq Id(\mathcal{A}_9)$.

		If there exists an evaluation of a monomial of $f$ such that
		$$
		(1\otimes e_1)(h_1\otimes j_1)(g_1\otimes ce_2)(h_2\otimes j_2)(1\otimes e_1) \ne 0,
		$$ for some $h_1\otimes j_1,h_2\otimes j_2, \in G^{(0)}\otimes  J^{(0)}\cup G^{(1)} \otimes  J^{(1)}$ and $ g_1 \in G^{(1)}$,
		then
		$
		e_1j_1ce_2j_2e_1 \ne 0.
		$
		If we put $j'_1 = j_1c$
		we are done by the previous case.

		Next let us consider  the case when there exists an evaluation of a monomial of $f$ such that
		$$
		(g_1\otimes e_2)(h_1\otimes j_1)(1\otimes e_1)(h_2\otimes j_2)(g_2\otimes e_2)\ne 0,
		$$
		for some $h_1\otimes j_1,h_2\otimes j_2, \in G^{(0)}\otimes  J^{(0)}\cup G^{(1)} \otimes  J^{(1)}, g_1, g_2 \in G^{(0)}$.
		By Lemma \ref{lemmaA_5}, there exists a $\mathbb{Z}_2$-graded subalgebra $\bar B$ of $B$  such that $Id(\mathcal{A}_5) \supseteq Id(G(\bar B))$.

		Hence we may assume that if $\bar m\ne 0$ is the evaluation in $\bar f$ of a monomial of $f$,
		$(1\otimes e_i)\bar m (1\otimes e_i)=0$, $1\le i\le 2$. This also says that
		$(1 \otimes e_i) \bar f (1 \otimes e_i) =0$ and $(1 \otimes e_i) \bar f= \bar f (1 \otimes e_i) =0$, $1\le i\le 2$.
		
		Next suppose that a monomial of $f$ has an evaluation of the type
		$$
		(h_1 \otimes j_1)(1\otimes e_1)(h_2\otimes j_2) (g_1\otimes e_2)
		(h_3\otimes j_3)\ne 0,
		$$
		for some $h_1\otimes j_1,h_2\otimes j_2,h_3\otimes j_3\in G^{(0)}\otimes  J^{(0)}\cup G^{(1)} \otimes  J^{(1)}, g_1 \in G^{(0)}$.

		Suppose $j_1e_1j_2e_2j_3\ne 0$. By Lemma \ref{lemmaA_6} there exists a $\mathbb{Z}_2$-graded subalgebra $\bar B$ of $B$  such that $Id(G(\bar B))\subseteq
		Id(\mathcal{A}_6)$.  The same result holds when
		$j_1e_1j_2ce_2j_3\ne 0$.

		Now we consider the case when a monomial of $f$ has an evaluation of the type
		$$
		(h_1 \otimes j_1)(g_1\otimes e_2)(h_2\otimes j_2) (1\otimes e_1)
		(h_3\otimes j_3)\ne 0,
		$$
		for some $h_1\otimes j_1,h_2\otimes j_2,h_3\otimes j_3\in G^{(0)}\otimes  J^{(0)}\cup G^{(1)} \otimes  J^{(1)}$ and $g_1 \in G^{(0)}.$
		By Lemma \ref{lemmaA_7} there exists a $\mathbb{Z}_2$-graded subalgebra $\bar B$ of $B$  such that $Id(G(\bar B))\subseteq Id(\mathcal{A}_7)$.  The same result
		holds when
		$j_1ce_2j_2e_1j_3\ne 0$.

		As remarked before,  if $\bar m\ne 0,$ then  $(1 \otimes e_i) \bar m (1 \otimes e_i) =0$ and also  $(1 \otimes e_i) \bar f= \bar f (1 \otimes e_i) =0$, $\forall
		i=1,2$. If we denote $(1 \otimes e_i)=E_i$, then  we may write, as in the previous case,
		\begin{equation} \label{eq6}
			\bar f = E_1\bar f_1+ \bar f_2 E_1+ E_2 \bar f_3 + \bar f_4 E_2.
		\end{equation}
		From $E_1\bar f=0$ we get
		$E_1\bar f_1  + E_1\bar f_4 E_2= 0$
		and so
		$E_1\bar f_1 E_2 + E_1\bar f_4 E_2= 0.$
		It follows that $E_1\Bar{f}_1 = E_1\Bar{f}_1E_2.$
		
		Moreover, since $\bar fE_2 =0$, from (\ref{eq6}) similarly we get
		$
		E_1\bar f_1E_2+ \bar f_4 E_2= 0
		$
		and so
		$
		E_1\bar f_1+\bar f_4 E_2= 0.
		$
		As a consequence we obtain that
		$$
		\bar f =  \bar f_2 E_1+ E_2 \bar f_3.
		$$
		Since
		$0= E_2\bar f =  E_2\bar f_2 E_1+ E_2 \bar f_3,$
		we get $\bar f =  \bar f_2 E_1+ E_2 \bar f_3= \bar f_2 E_1 - E_2\bar f_2 E_1$.
		It follows that
		$0=\bar f E_1= \bar f_2 E_1 - E_2\bar f_2 E_1 = \bar f \neq 0$,  a contradiction.
	\end{proof}

	Now we are able to characterize the varieties with proper central exponent equal to two.
	
	\begin{Corollary}
		Let $\mathcal{V}$ be a variety of PI-algebras over a field $F$ of characteristic zero. Then  $exp^\delta(\mathcal V)=2$ if and only if  $A_i\not \in
		\mathcal{V}$ for every $i\in \{1, \ldots , 9\}$ and  $D \in \mathcal{V}$ or $D_0 \in \mathcal{V}$ or $G \in \mathcal{V}$.  \end{Corollary}
	\begin{proof}
		By \cite{GLP} exp$^\delta (\mathcal V)\geq 2$ if and only if $D \in \mathcal V$ or $D_0 \in \mathcal V$ or $G \in \mathcal V$. Moreover, by Theorem
		\ref{teorema1}, exp$^\delta (\mathcal V)\leq 2$ if and only if $A_i \not \in \mathcal V$ for $i\in \{1, \ldots ,9\}$ and so the proof is completed.
	\end{proof}
	
	\medskip
	Next we recall the following.
	
	\begin{Definition}
		Let $\mathcal{V}$ be a variety of algebras. We say that $\mathcal{V}$ is minimal of proper central exponent $d\geq 2$ if exp$^\delta(\mathcal{V})=d$ and for
		every proper subvariety $\mathcal{U}\subset \mathcal{V}$ we have that exp$^\delta(\mathcal{U})<d$.
	\end{Definition}
	
	\medskip
	
	We denote $var(\mathcal{A}_i)=\mathcal{V}_i$, for $1\le i\le 9$.

	\medskip

	\begin{Proposition} \label{differentvarieties}
		For  $1\le i,j\le 7$, $i\neq j$, we have that  $\mathcal{V}_i \not \subseteq \mathcal{V}_j$.
		Moreover $\mathcal{V}_8, \mathcal{V}_9 \not \subseteq \mathcal{V}_i$, $1\le i\le 7$.
	\end{Proposition}
	\begin{proof}
		Since $exp(\mathcal{V}_1)=exp(\mathcal{V}_2)=4$ and, for all $i \ge 3$,   $exp(\mathcal{V}_i)=3$,
		then, for $j=1,2$, $\mathcal{V}_j \not \subseteq \mathcal{V}_i$.
		
		Since $St_4,$ the standard polynomial of degree four, is an identity of  $\mathcal{A}_1$ but   $St_4 \not\in Id(\mathcal{A}_i)$, for all $i \ge 3$, then
		$\mathcal{V}_i \not \subseteq \mathcal{V}_1$.
		Moreover, $[[x_1,x_2]^2,x_3] \in Id(\mathcal{A}_1)$ but $[[x_1,x_2]^2,x_3] \not \in Id(\mathcal{A}_2)$ hence $\mathcal{V}_2 \not \subseteq \mathcal{V}_1$.
		Also, $[[x_1,x_2],[x_3,x_4],x_5]\in Id(\mathcal{A}_2)$ but it is not in $Id(A_i)$, $i =3, \ldots , 9$, then $\mathcal{V}_i \not \subseteq \mathcal{V}_2$. Since
		$[[x_1,x_2],[x_3,x_4],x_5]\in Id(\mathcal{A}_2)$ and
		$[[x_1,x_2],[x_3,x_4],x_5]\not \in Id(\mathcal{A}_1)$, then  $\mathcal{V}_1 \not \subseteq \mathcal{V}_2$.
		
		The proof is completed by an easy check of the following  statements.
		
		\begin{itemize}
			\item[-]
			
			$[x_1,x_2][x_3,x_4][x_5,x_6][x_7,x_8] \in Id(\mathcal{A}_3)$ but
			$\not \in Id(\mathcal{A}_i)$, $4\le i\le 9$,
			\smallskip

			\item[-]
			
			$x_1[x_2,x_3][x_4,x_5][x_6,x_7]x_8 \in Id(\mathcal{A}_4)$ but $\not \in  Id(\mathcal{A}_i)$, $3\le i\le  9, \ i \ne 4 $
			\smallskip
			
			\item[-]
			$[x_1,x_2, x_3][x_4,x_5][x_6,x_7,x_8] \in Id(\mathcal{A}_5)$ but
			$ \not \in Id(A_i)$,  $3\le i\le 9$, $i \neq 5$,
			
			\smallskip
			
			\item[-]
			$x_1[x_2,x_3][x_4,x_5, x_6]x_7 \in Id(\mathcal{A}_6)$ but $  \not \in   Id(\mathcal{A}_i)$ $i=3, \dots , 9$, $i \neq 6$,
			
			\smallskip
			
			\item[-] $x_1[x_2,x_3, x_4][x_5, x_6]x_7 \in Id(\mathcal{A}_7)$ but $  \not \in Id(\mathcal{A}_i)$,  $3\le i\le 9$, $i \neq 7$.
		\end{itemize}
	\end{proof}
	
	As a consequence we get the following
	
	\begin{Corollary} The varieties  $\mathcal{V}_1$ and $\mathcal{V}_2$ are minimal  of proper central exponent equal to $4.$
		The varieties $\mathcal{V}_i,$  $3\le i\le 7$, are minimal of proper central exponent equal to $3.$
	\end{Corollary}
	\begin{proof} By Lemma \ref{exponent Ai},
		exp$^\delta(\mathcal{V}_1)=$ exp$^\delta(\mathcal{V}_2)=4$ and  exp$^\delta(\mathcal{V}_i)=3$, for  $3\le i\le 7$. Now,
		let $\mathcal{U}$ be a proper subvariety of $\mathcal{V}_i$ for some $i \in \{ 1, \ldots , 7\}$,
		and suppose  that exp$^\delta(\mathcal{U})=$ exp$^\delta(\mathcal{V}_i)> 2$.
		
		By Theorem \ref{teorema1}, there exist $j \in \{1, \ldots , 9\}$ such that $\mathcal{A} _j \in \mathcal{U}$.
		Hence $ \mathcal{V}_j\subseteq \mathcal{U}\subset \mathcal{V}_i$ and by  Proposition \ref{differentvarieties},  we obtain that
		$\mathcal{V}_j=\mathcal{U}=\mathcal{V}_i$, a contradiction.
	\end{proof}

\end{document}